\documentclass[12pt]{amsart}

\pdfoutput=1
\usepackage{latexsym}
\usepackage[centertags]{amsmath}
\usepackage{amsfonts}
\usepackage{amssymb}
\usepackage{amsthm}
\usepackage{newlfont}
\usepackage{graphics}
\usepackage{color}
\usepackage{float}
\usepackage{wrapfig}

\usepackage[usenames,dvipsnames]{xcolor}
\usepackage{graphicx}

\usepackage{wrapfig}

\newtheorem{theo}{Theorem}
\newtheorem{defn}[theo]{Definition}
\newtheorem{exam}[theo]{Example}
\newtheorem{lem} [theo]{Lemma}
\newtheorem{cor}[theo]{Corollary}
\newtheorem{prop}[theo]{Proposition}

\newtheorem{algor}[theo]{Algorithm}

\makeatletter \@addtoreset{equation}{section}
\@addtoreset{theo}{}\makeatother

\newcommand{\rank}{\operatorname{rank}}

\title{A symmetric chain decomposition of $N(m,n)$ of composition}

\author{Yueming Zhong}

 \address{ School of Mathematical Sciences, Capital Normal University,
 Beijing 100048, PR China}
 \email{\texttt{zhongyueming107@gmail.com}}
\date{April 22, 2021}

\begin{document}

\begin{abstract}
A poset is called a symmetric chain decomposition if the poset can be expressed as a disjoint union of symmetric chains.
For positive integers $m$ and $n$, let $N(m,n)$ denote the set of all compositions $\alpha=(\alpha_1,\cdots,\alpha_m)$, with $0\le \alpha_i \le n$ for each $i=1,\cdots,m$. Define order $<$ as follow, $\forall \alpha,\beta \in N(m,n)$, $\beta < \alpha$ if and only if $\beta_i \le \alpha_i(i=1,\cdots,m)$ and $\sum\limits_{i=1}^{m}\beta_i <\sum\limits_{i=1}^{m}\alpha_i$. In this paper, we show that the poset $(N(m,n),<)$ can be expressed as a disjoint of symmetric chains by constructive method.
\end{abstract}

\maketitle

\noindent
\begin{small}
 \emph{Mathematic subject classification}: Primary 05A19; Secondary 05E99.
\end{small}

\noindent
\begin{small}
\end{small}

\section{Introduction}
For positive integers $m$ and $n$, let $N(m,n)$ denote the set of all compositions $\alpha=(\alpha_1,\cdots,\alpha_m)$, with $0\le \alpha_i \le n$ for each $i=1,\cdots,m$. Define order $<$ as follow, $\forall \alpha,\beta \in N(m,n)$, $\beta < \alpha$ if and only if $\beta_i \le \alpha_i(i=1,\cdots,m)$ and $\sum\limits_{i=1}^{m}\beta_i <\sum\limits_{i=1}^{m}\alpha_i$.

For $\alpha \in N(m,n)$, define: $rank(\alpha)=\sum\limits_{i=1}^{m}\alpha_i$.

Let $(P,<)$ be a graded poset of rank $n$. We say that these elements $x_1, x_2, \cdots, x_k$ of $P$ has a symmetric chain if

(1) $x_{i+1}$ cover $x_i$, $i<k-1$;

(2) $rank(x_{i+1}) + rank(x_i)= rank(P)$, where $rank(P)$ is the maximum rank.

A poset is called a symmetric chain decomposition if the poset can be expressed as a disjoint union of symmetric chains.
Many beautiful results have been derived by symmetric chain since de Bruijn introduced it in 1949, see \cite{DeBruijn}.  The elegant proof of Lubell appears in \cite{Lubell}. Symmetric chain decompositions requires each chain is rank symmetric. Its existence for $L(m,n)$ was not confirmed for $m\ge 5$.
Symmetric chain decompositions was only constructed for $m=3$ \cite{B.Lindstrom,Wen} and for $m=4$ \cite{Nathan,Wen}. A kind of chain decomposition of $L(m,n)$ for $m=3,4$ is also given in \cite{GuoceXin}. For more related symmetry chain decomposition, see \cite{Dhand-necklacePosets,Hersh-Boolean,Jordan-necklacePoset,Muhle-noncrossingPartition}.


In this paper, we show that the poset $(N(m,n),<)$ can be expressed as a disjoint of symmetric chains by constructive method, as the following Theorem \ref{theo2}.

\begin{theo}\label{theo2}
For any positive integers $m,n$, the poset $(N(m,n),<)$ has a symmetric chain decomposition into $\{C_\alpha: \alpha\in S_{m,n}\}$.
That is, i) $N(m,n)=\cup_{\alpha\in S_{m,n}} C_{\alpha}$; ii) the above is a disjoint union; iii) $\rank(S(C_\alpha))+\rank(E(C_\alpha))=mn$.
(Note: $C_\alpha$, $S(C_\alpha)$ and $E(C_\alpha)$ as defined in \ref{Def:Calpha}).
\end{theo}

This paper is organized as follows:

In Section \ref{sec:Nmn}, firstly, we give the construction of $C_\alpha$ as Algorithm \ref{alg-CTabCalpah}.
The feasibility of our construction is discussed in the second subsection.
Finally, we prove that these $C_\alpha$ do give a symmetric chain decomposition of $N(m,n)$. This result is given by Theorem \ref{theo2}.
In Section \ref{sec:involution}, we give an beautiful property of $C_{\alpha}$: an induced involution as Theorem \ref{theo:involution}.

\section{The symmetric chain decomposition of the $N(m,n)$\label{sec:Nmn}}
In this section, we will give a symmetric chain decomposition of $N(m,n)$ into $C_\alpha$.
This is divided into 3 subsections. Firstly, we give the construction of $C_\alpha$.
The feasibility of our construction is discussed in the second subsection.
Finally, we prove that these $C_\alpha$ do give a symmetric chain decomposition of $N(m,n)$.

\subsection{The construction of the symmetric chains $C_{\alpha}$}\label{Def:Calpha}
The symmetric chains $C_\alpha$ are indexed by $S_{m,n}$, which is defined by
\begin{align*}
        S_{m,n}=&\{\alpha\in N(m,n): \alpha=(\alpha_1,\cdots,\alpha_{m-1},0), \alpha\models k, 0\leq k \leq \lfloor\frac{mn}{2}\rfloor, \\
        &\qquad \sum\limits_{i=t}^{m-1}\alpha_i \le \sum\limits_{i=t+1}^{m}(n-\alpha_i),1\le t \le m-1\}.
\end{align*}

For each ordered partition $\alpha\in S_{m,n}$, we construct a symmetric chain tableau
$C_\alpha$ as follows.

\begin{algor}
\label{alg-CTabCalpah}\normalfont

\noindent
\textbf{Input:} An ordered partition $\alpha \in S_{m,n}$.

\noindent
\textbf{Output:} The chain tableau $C_\alpha$.

\begin{enumerate}
  \item[\textbf{Step 1.}] Draw an empty tableau $T$ of shape $m\times n$.

   \item[\textbf{Step 2.}] For each $i$, color the first $\alpha_i$ cells in the $i$-th row of $T$ by green. These cells are called \emph{fixed cells} for $\alpha$.

    \item[\textbf{Step 3.}]  For $i$ from $1$ to $m-1$, we successively color some cells by gray as follows and call them \emph{forbidden cells} for $\alpha$.

                                           \begin{enumerate}
                                             \item For $i=1$, read the unfixed cells starting at row $i+1=2$, from right to left, and top to bottom. Color the first $\alpha_1$ cells by gray.
                                             These are the \emph{forbidden cells} for $\alpha_1$.

                                             \item Suppose the forbidden cells for $\alpha_1,\dots, \alpha_{i-1}$ have been colored by gray
                                             and we are going to construct the forbidden cells for $\alpha_i$. Starting at row $i+1$, we read the unfixed and unforbidden cells from right to left and from top to bottom. Color the first $\alpha_i$ cells by gray. These are the \emph{forbidden cells} for $\alpha_i$.

                                             \item Repeat the above step until the forbidden cells for $\alpha_{m-1}$ are colored gray.
                                           \end{enumerate}

 \item[\textbf{Step 4.}]     Call the unfixed and unforbidden cells fillable cells. Successively fill $1,2,3,\dots$ in the fillable cells from left to right and from bottom to top.
                           The resulting tableau is the chain tableau $C_\alpha$ of $\alpha$.
\end{enumerate}
\end{algor}
                           The number of \emph{forbidden cells} in the $i$-th row of $C_\alpha$  is denoted $\alpha^E_i$ for each $i=1,\cdots,m$. Clearly $\alpha^E_1=0$. Let $\alpha^E=(\alpha^E_1,\alpha^E_2,\cdots,\alpha^E_m)$. It is convenient to denote by $S(C_\alpha)=\alpha$ the starting point of $C_{\alpha}$ and by $E(C_\alpha)=
                           (n-\alpha^E_1,n-\alpha^E_2,\cdots,n-\alpha^E_m)$ the end point of $C_{\alpha}$.  The set $S_{m,n}$ is called the \emph{starting set} of our chain decomposition of $N(m,n)$.

For the sake of clarity, we give two examples to illustrate the Algorithm \ref{alg-CTabCalpah}.
                         \begin{exam}\label{exa1}
                                 The construction of $C_\alpha$ for $\alpha=(2,0,5,0)\in S_{4,6}$ is given in Figure \ref{fig:exa1}. The tableau $C_\alpha$ has natural correspondence with a situated chain, still denoted $C_{\alpha}:(2,0,5,0)<(2,0,5,1)<(2,0,6,1)<(2,1,6,1)<(2,2,6,1)<(2,3,6,1)<(2,4,6,1)<(3,4,6,1)<(4,4,6,1)<(5,4,6,1)<(6,4,6,1)$.
                                 The chain starts at $S(C_\alpha)=(2,0,5,0)$ and ends at $E(C_\alpha)=(6,4,6,1).$ We get $\alpha^E=(0,2,0,5)$.
                                \begin{figure}[!ht]
                                \centering{
                                \includegraphics[height=1 in]{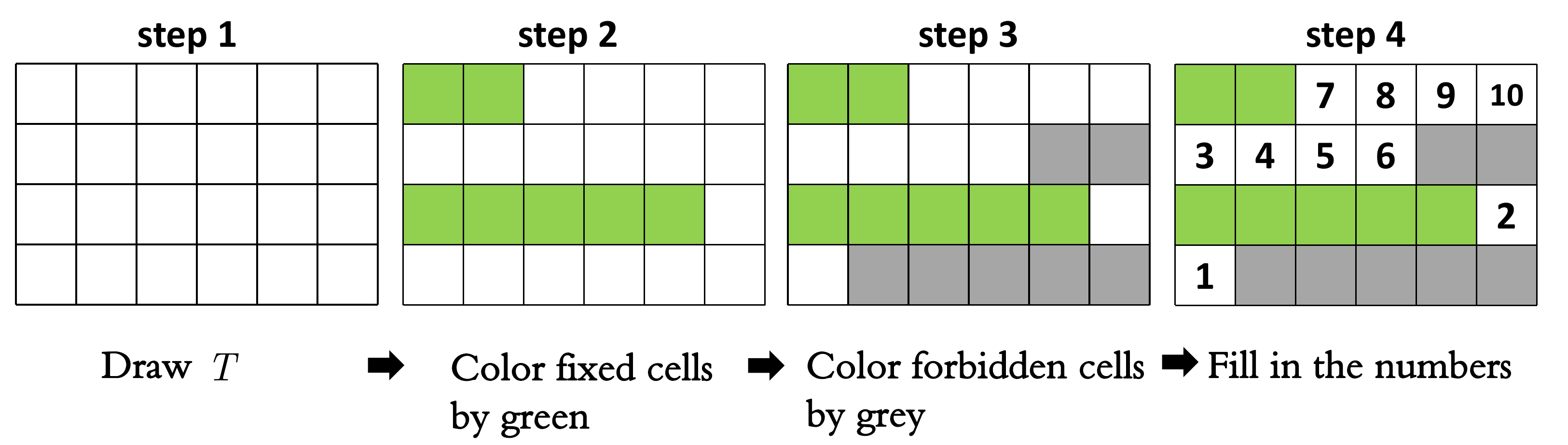}}
                                \caption{The construction of $C_{(2,0,5,0)}$.\label{fig:exa1}}
                                \end{figure}
                         \end{exam}
                         \begin{exam}
                                The construction of $C_\alpha$ for $\alpha=(1,3,2,0)\in S_{4,4}$ is given in Figure \ref{fig:exa2}. The tableau $C_\alpha$ has natural correspondence with a situated chain, still denoted $C_{\alpha}:(1,3,2,0)<(1,3,2,1)<(2,3,2,1)<(3,3,2,1)<(4,3,2,1)$.
                                The chain starts at $S(C_\alpha)=(1,3,2,0)$ and ends at $E(C_\alpha)=(4,3,2,1).$ We get $\alpha^E=(0,1,2,3)$.
                                \begin{figure}[!ht]
                                \centering{
                                \includegraphics[height=1 in]{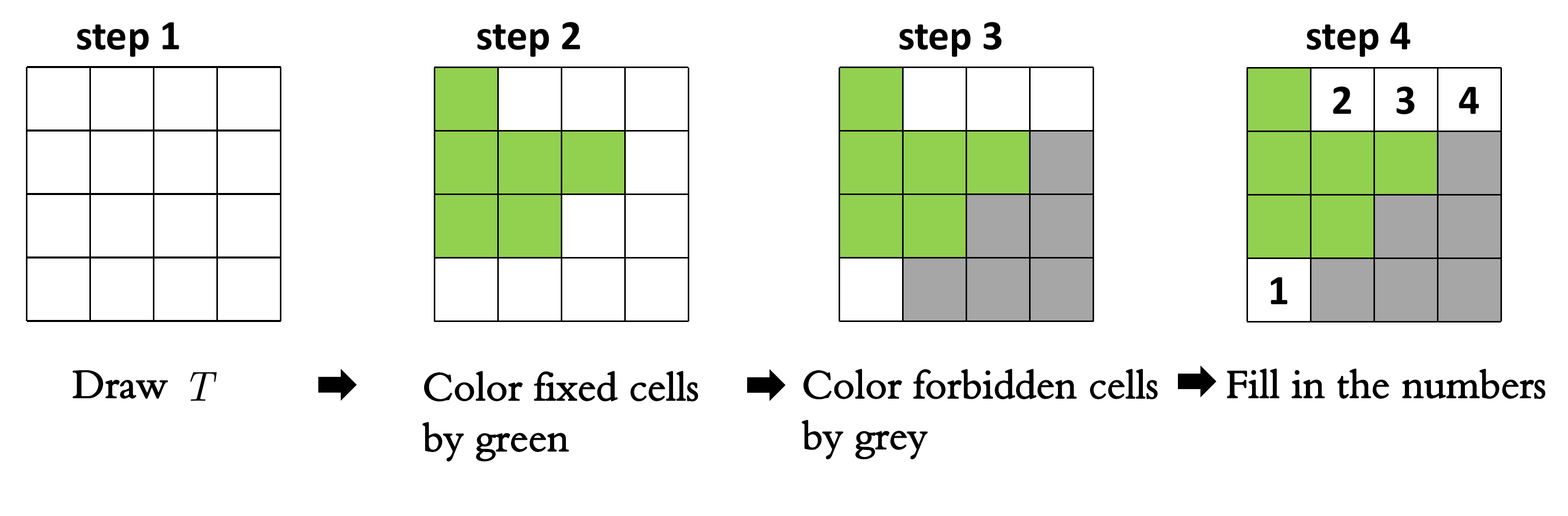}}
                                \caption{The construction of $C_{(1,3,2,0)}$.\label{fig:exa2}}
                                \end{figure}
                               \end{exam}


\subsection{The feasibility of Algorithm \ref{alg-CTabCalpah}}
To show the feasibility of Algorithm \ref{alg-CTabCalpah}, it suffices to show the third step is feasible, since the other steps are obviously feasible.
This is achieved by the following Lemma \ref{lem:splitting}. We need the following definition first.

 \begin{defn}
                                       Let $\alpha\in S_{m,n}$. We recursively define the splitting rows for $\alpha$ as follows.

                                       Row 1 is splitting.

                                       If row $p$ is splitting, then row $q+1$ is splitting, where $q$ is the smallest positive integer with $p\le q\le m-1$ such that $\sum\limits_{i=p}^{q}\alpha_i\le \sum\limits_{i=p+1}^{q+1}(n-\alpha_i)$, which condition is satisfied for $q=m-1$ for any $p$ by
                                       definition of $\alpha$.
 \end{defn}
The following lemma justifies the name of splitting row.

\begin{lem}\label{lem:splitting}
Suppose $\alpha \in S_{m,n}$ has splitting row indices $1=q_{1}<q_{2}<\cdots<q_{s}=m$.
According Algorithm \ref{alg-CTabCalpah}, if row $i$ is not splitting, then the forbidden cells for $\alpha_{i-1}$ has at least one cell in row $i+1$; otherwise
if row $i$ is splitting, then the forbidden cells for $\alpha_{i-1}$ has no cell in row $i+1$.
\end{lem}
\begin{proof}
We prove by induction on $k$ the following claim: the lemma holds for $i=1,2,\dots, q_k$.

The base case $k=1$ is simple: since there is no forbidden cells for $\alpha_{0}$ and row $1$ is always splitting.
Assume the claim holds true for $k$, and we want to show that it holds for $k+1$.

Now row $i=q_k$ is splitting. By induction hypothesise, there is no forbidden cells for $\alpha_{q_k-1}$ in row $q_k+1$.
The forbidden cells for $\alpha_i$ starts at the rightmost cell in row $i+1$.
Then row $i+1$ is splitting if and only if $n-\alpha_{i+1}\ge \alpha_i$. This shows the claim holds when $q_{k+1}=i+1$.
If row $i+1$ is not splitting, then $n-\alpha_{i+1}> \alpha_i$, so that the forbidden cells for $\alpha_i$ occupies all
the rightmost $n-\alpha_{i+1}$ cells and at least one cell in row $i+2$.
This gives $\alpha_{i+1}^E=n-\alpha_{i+1}$.

By similar reasons, the forbidden cells for $\alpha_j$ has at least one
cell in row $j+2$ for $q_k \le q_{k+1}-2$, and we have
$\alpha^E_j=n-\alpha_j$ for $j=q_{k}+1,q_{k}+2,\cdots,q_{k+1}-1$.
Next consider the forbidden cells for $\alpha_i$ with $i=q_{k+1}-1$. They starts in row $i+2$ and end in the same row,
because of the inequality $\alpha_{q_k} + \alpha_{q_k+1} +\cdots + \alpha_{q_{k+1}-1} \le (n-\alpha_{{q_k}+1}) + (n-\alpha_{{q_k}+2}) + \cdots + (n-\alpha_{q_{k+1}})$
by the definition of splitting row $q_{k+1}$. The claim then holds for $k+1$.
\end{proof}

\begin{cor}\label{cor:aEnd}
Suppose $\alpha \in S_{m,n}$ has splitting row indices $1=q_{1}<q_{2}<\cdots<q_{s}=m$.
Then the following conditions hold for all $1\le k<s$.                                           \[
                                           \begin{cases}
                                            \alpha_{q_k} > n-\alpha_{{q_k}+1} ,\\
                                            \alpha_{q_k} + \alpha_{q_k+1} > (n-\alpha_{{q_k}+1}) + (n-\alpha_{{q_k}+2}) ,\\
                                           \qquad   \qquad \vdots\\
                                           \alpha_{q_k} + \alpha_{q_k+1} +\cdots + \alpha_{q_{k+1}-2} > (n-\alpha_{{q_k}+1}) + (n-\alpha_{{q_k}+2}) + \cdots + (n-\alpha_{q_{k+1}-1}) ,\\
                                           \alpha_{q_k} + \alpha_{q_k+1} +\cdots + \alpha_{q_{k+1}-1} \le (n-\alpha_{{q_k}+1}) + (n-\alpha_{{q_k}+2}) + \cdots + (n-\alpha_{q_{k+1}}) .\\
                                           \end{cases}\]
Furthermore, we have $\alpha^E_j=n-\alpha_j$ for each $j=q_{k}+1,q_{k}+2,\cdots,q_{k+1}-1$ and $\alpha^E_{q_{{k+1}}}=\sum\limits_{j=q_{k}}^{q_{k+1}-1}\alpha_j-\sum\limits_{j=q_{k}+1}^{q_{{k+1}}-1}\alpha^E_j$.
\end{cor}

Corollary \ref{cor:aEnd} provides a better way to compute $\alpha^E$. See Figure \ref{fig:alphaE} for an example.

                 \begin{figure}[!ht]
                 \centering{
                 \includegraphics[height=3 in]{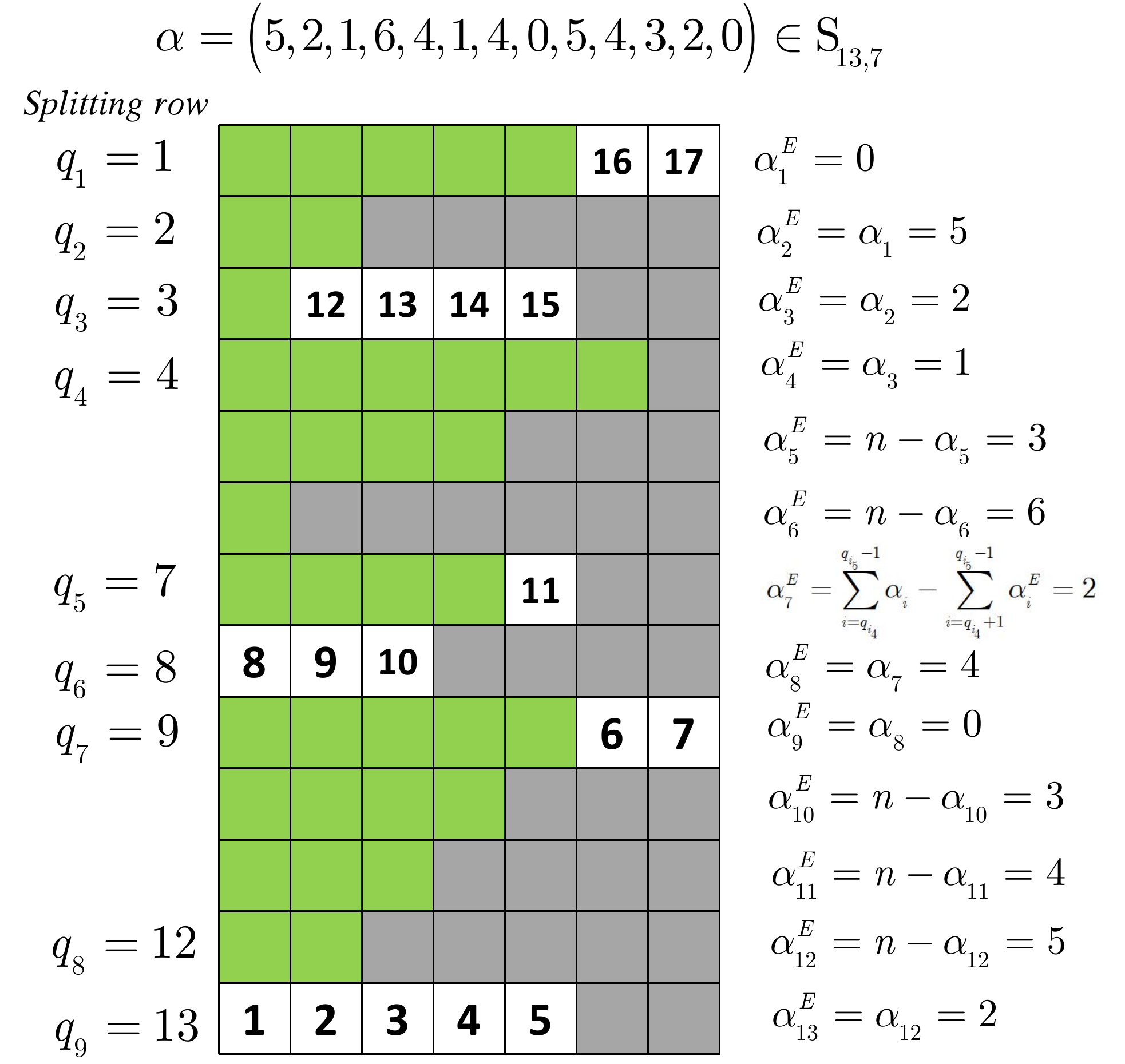}}
                 \caption{Computation of $\alpha^E$ by Corollary \ref{cor:aEnd}, where $\alpha=(5,2,1,6,4,1,4,0,5,4,3,2,0)$.\label{fig:alphaE}}
                 \end{figure}

\subsection{The symmetric chain decomposition of $N(m,n)$}
Given $\alpha\in S_{m,n}$. Elements $c=(c_1,c_2,\cdots,c_m)\in C_{\alpha}$ in $C_\alpha$ can be described by the vector
$a=(a_1,a_2,\cdots,a_m)=c-\alpha$, called a \emph{fillable vector}. We will use the following equivalent definition, which follows easily from Algorithm
\ref{alg-CTabCalpah}: Let $d^{\alpha}_i=(n-\alpha_i)-\alpha^E_i$ be the number of cells in row $i$ that are neither fixed nor forbidden.
A vector $a_{\alpha}=(a_1,a_2,\cdots,a_m)$ is a {fillable vector} with respect to $\alpha$ if and only if there exists an integer $0 \le i_0 \le m$ such that $a_i=0$ for each $i=1,\cdots, i_0-1$, $0<a_{i_0}< d^{\alpha}_{i_0}$ and $a_j=d^{\alpha}_j$ for each $j=i_0+1,\cdots,m$. The $i_0=0$ case corresponds to $a_i=d^{\alpha}_i$ for $i=1,\cdots,m$.

According to the Algorithm \ref{alg-CTabCalpah}, we have the following properties:
\begin{prop}\label{prop1}
Assume $c=(c_1,c_2,\cdots,c_m)\in C_{\alpha}$ for some $\alpha=(\alpha_1,\cdots,\alpha_{m-1},0)\in S_{m,n}$. Then
there exists a unique fillable vector $a=(a_1,a_2,\cdots,a_m)$ such that $c=\alpha+a$. Moreover,
denote by $P=\{1\le i\le m: a_i>0\} \cup \{m\}$. Then

i) If $i\in P$ then row $i$ is splitting, and $c_t+\alpha^E_t=n$ for $t\ge i+1$;

ii) Suppose $j\in P$ and $a_t=0$ for all $i+1\le t \le j-1$. Then we have
\begin{align}
  \sum\limits_ {t=i}^{j-1}\alpha_t=\sum\limits_{t=i+1}^{j}\alpha^E_t=\sum\limits_{t=i+1}^{j}(n-c_t), \qquad \text{ if } a_i>0;\\
  \sum\limits_{t=i}^{j-1}\alpha_t \le \sum\limits_{t=i+1}^{j}\alpha^E_t=\sum\limits_{t=i+1}^{j}(n-c_t), \qquad \text{ if } a_i=0.
\end{align}
\end{prop}

Next, we show that the $\biguplus_{\alpha\in S_{m,n}}C_{\alpha}$ by the above Algorithm \ref{alg-CTabCalpah} is a symmetric chain decomposition of $N(m,n)$ as described by the Theorem \ref{theo2}. Now we give the proof of the Theorem \ref{theo2} as follows.

\begin{proof}
Part i) and ii) follows from Lemmas \ref{lem:part1} and \ref{lem:part2} below. Part iii) is obvious.
\end{proof}

\begin{lem}\label{lem:part2}
  Suppose $\alpha$ and $\beta$ are distinct element of $S_{m,n}$. Then $C_\alpha \cap C_\beta=\emptyset$.
\end{lem}
\begin{proof}
Write $\alpha=(\alpha_1,\alpha_2,\cdots,\alpha_{m-1},0)$ and $\beta=(\beta_1,\beta_2,\cdots,\beta_{m-1},0)$.
Since $\alpha\neq \beta$, there exists a positive integer $1 \le i_0 \le m-1$ such that $\alpha_{i_0}\neq\beta_{i_0}$ and $\alpha_{i}=\beta_{i}$ for each
$i$ with $i_0<i\le m-1$.   Without loss of generality, we may assume $\alpha_{i_0}<\beta_{i_0}$.

Assume to the contrary that $c\in C_{\alpha}\cap C_{\beta}$. Then there exists fillable vectors $a_{\alpha}=(a_1,\cdots,a_m)$ and $b_{\beta}=(b_1,\cdots,b_m)$ such that
$c=\alpha+a_{\alpha}=\beta+b_{\beta}$.

Now $a_{i_0}+\alpha_{i_0}=b_{i_0}+\beta_{i_0}$ implies that $a_{i_0}=b_{i_0}+\beta_{i_0}-\alpha_{i_0}>0$. It follows that
                                          \[
                                           \begin{cases}
                                           \text{$\textcircled{1}$\  $\alpha_{i}=\beta_{i}$ and $a_i=b_i$, where $i_0\le i\le m$.}\\
                                           \text{$\textcircled{2}\  X:=\alpha_{i_0}+\alpha_{i_0+1}+\cdots+\alpha_{m-1}$}\\
                                           \text{$=(n-\alpha_{i_0+1}-a_{i_0})+\cdots+(n-\alpha_{m-1}-a_{m-1})+(n-a_m)$, where $a_{i_0}>0$.}\\
                                           \text{$\textcircled{3}\  \beta_{i_0}+\beta_{i_0+1}+\cdots+\beta_{m-1}$}\\
                                           \text{$\le(n-\beta_{i_0+1}-b_{i_0})+\cdots+(n-\beta_{m-1}-b_{m-1})+(n-b_m)$.}
                                           \end{cases}\]
Thus we get
                                           $X=\alpha_{i_0}+\alpha_{i_0+1}+\cdots+\alpha_{m-1}<\beta_{i_0}+\beta_{i_0+1}+\cdots+\beta_{m-1}\le X$,
                                           where $X=(n-\alpha_{i_0+1}-a_{i_0})+\cdots+(n-\alpha_{m-1}-a_{m-1})+(n-a_m)=(n-\beta_{i_0+1}-b_{i_0})+\cdots+(n-\beta_{m-1}-b_{m-1})+(n-b_m)$.
                                           Therefore $X<X$, which is a contradiction. This proves the lemma.

\end{proof}

\begin{lem}\label{lem:part1}
  For any $c\in N(m,n)$, there is an $\alpha\in S_{m,n}$ such that $c\in C_\alpha$.
\end{lem}

\begin{proof}
For given $c=(c_1,c_2,\cdots,c_m)$, we assume the existence of an undetermined
$\alpha=(\alpha_1,\alpha_2,\cdots,\alpha_{m-1},0)\in S_{m,n}$ and an undetermined fillable vector $a=(a_1,a_2,\cdots,a_m)$ such that $c=\alpha+a$.

We recursively compute $P$, $\alpha$, $\alpha^E$, and $a$ backwardly as follows.
\begin{enumerate}
  \item Set $P=\{m\}$, $a_m=c_m$, $\alpha_m=0$, and $t=m$. Treat $a_0>0$.

  \item Iteratively using the following Claim to find a $r$ such that
  $a_{r+1}=a_{r+2}=\cdots = a_{t-1}=0$ and $a_r>0$ for some $r\le t-1$. Add the element $r$ into $P$.

  \item If $r>0$ then $a_{r+1}=a_{r+2}=\cdots = a_{t-1}=0$ which implies $\alpha_i=c_i$ for $r+1\le i \le t-1$,
  and set $\alpha_r= \sum\limits_{i=r+1}^{t}(n-c_i)-\sum\limits_{i=r+1}^{t-1}c_i$.
  Set $t=r$ and go to step 2.

   \item If $r=0$ then $a_{t-1}=a_{t-2}=\cdots = a_{1}=0$, and we have $\alpha_i=c_i$ for $1\le i \le t-1$.
\end{enumerate}
The lemma clearly follows if we prove the following claim.

\textbf{Claims:}
  Suppose $t\in P$, $r<t$ and $a_{r+1}=\cdots=a_{t-1}=0$. Then

i) if $\sum\limits_{i=r}^{t-1}c_i\le \sum\limits_{i=r+1}^{t}(n-c_i)$ then $a_r=0$;

ii) if $\sum\limits_{i=r}^{t-1}c_i> \sum\limits_{i=r+1}^{t}(n-c_i)$ then $a_r>0$. Consequently
$\alpha_i=c_i$ for each $i=r+1,\cdots,t-1$ . Hence $\sum\limits_{i=r}^{t-1}\alpha_i = \sum\limits_{i=r+1}^{t}(n-\alpha_i)$.

For part i), assume to the contrary that $a_r>0$ holds for possible $\alpha$ and $a$.
               By Property \ref{prop1}(ii)(1), we get $\alpha_r+\cdots+\alpha_{t-1} = \alpha^E_{r+1}+\cdots+\alpha^E_t$.
               Since $a_{r+1}=\cdots=a_{t-1}=0$, we have $\alpha_i=c_i$, and $\alpha^E_i=n-c_i$ for each $i=r+1,\cdots,t-1$. Note that $\alpha^E_t=n-c_t$.
               It follows that
               $$\sum\limits_{i=r}^{t-1}c_i\le (n-c_{r+1})+\cdots+(n-c_t) =\alpha_r+c_{r+1}+\cdots+c_{t-1}<c_r+\cdots+c_{t-1}. $$
This is a contradiction. Therefore $a_r=0$.

For part ii), we have $\alpha_i=c_i$ and $\alpha^E_i=n-c_i$ for each $i=r+1,\dots,t-1$.
Assume to the contrary that $a_r=0$ holds for possible $\alpha$ and $a$.
Then $\alpha_r=c_r$ and $\alpha^E_r=n-c_r$.
               By Property \ref{prop1}(ii)(2), we get $\alpha_r+\cdots+\alpha_{t-1} \le \alpha^E_{r+1}+\cdots+\alpha^E_t$,
which is rewritten as
$$c_r+\cdots+c_{t-1} \le \alpha^E_{r+1}+\cdots+\alpha^E_t\le (n-c_{r+1})+\cdots+(n-c_t).$$
This contradicts the hypothesis. Therefore $a_r>0$. Then by Corollary \ref{cor:aEnd}, we should have
$\sum\limits_{i=r}^{t-1}\alpha_i = \sum\limits_{i=r+1}^{t}(n-\alpha_i)$. Solving for $\alpha_r$ gives
 $\alpha_r=\sum\limits_{i=r+1}^{t}(n-c_i)-\sum\limits_{i=r+1}^{t-1}c_i$. Then $a_r=c_r-\alpha_r$. Note that $\alpha^E_r$
 cannot be determined yet.

\end{proof}

\begin{exam}
For $c=(5,2,3,6,4,1,5,3)\in N(8,7)$, we can find the chain $C_{\alpha}$ with $c\in C_{\alpha}$, where $\alpha=(5,2,1,6,4,1,4,0)$, as follows in Figure \ref{fig:cToalpha}.
\begin{figure}[H]
\centering{
\includegraphics[height=2.3 in]{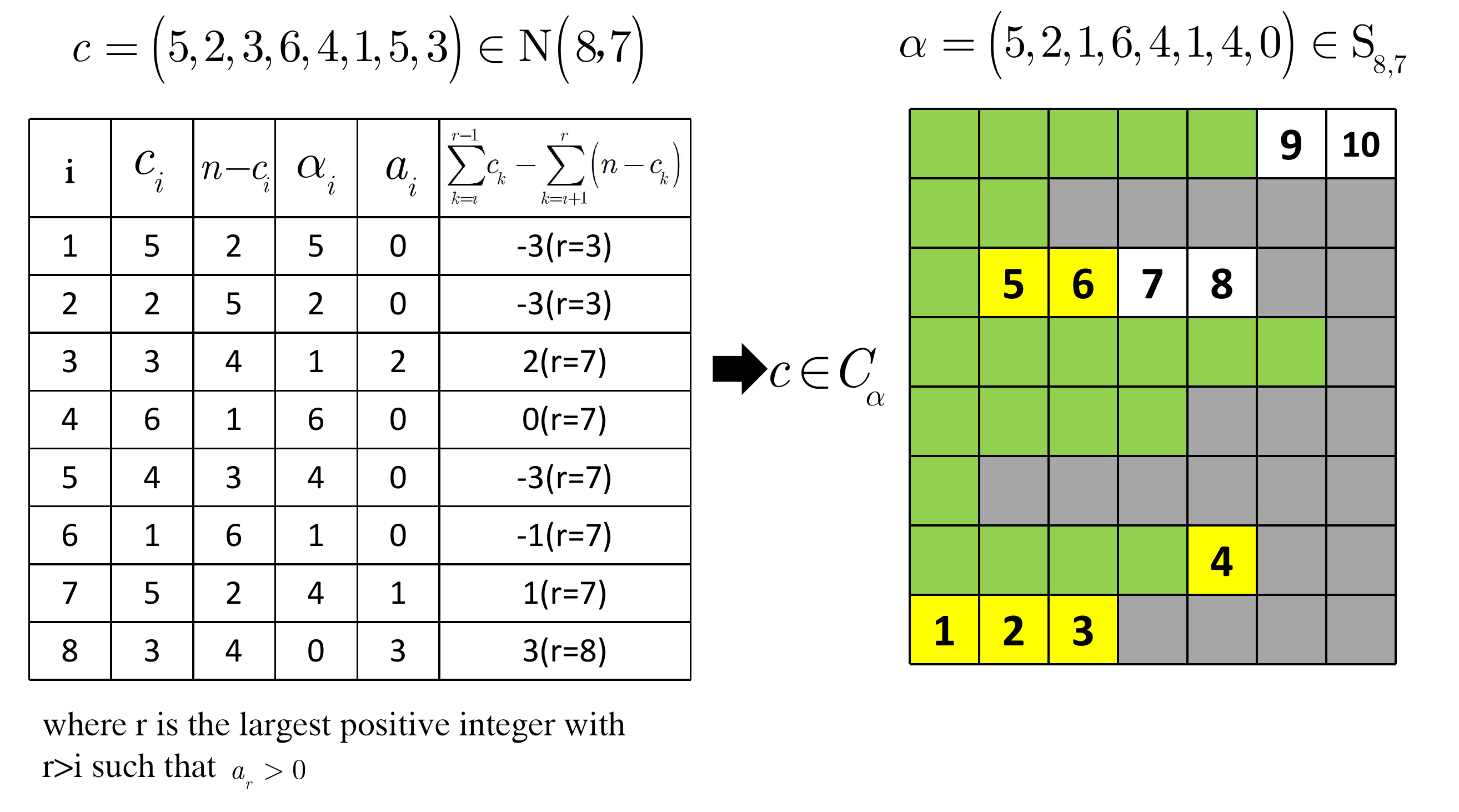}}
\caption{Computing $\alpha$ from $c$.\label{fig:cToalpha}}
\end{figure}

\end{exam}

\begin{exam}
               We give the symmetric chain decomposition of $N(3,2)$ by Algorithm \ref{alg-CTabCalpah} in Figure \ref{fig:N32} as follows.
\begin{figure}[!ht]
\centering{
\includegraphics[height=1.7 in]{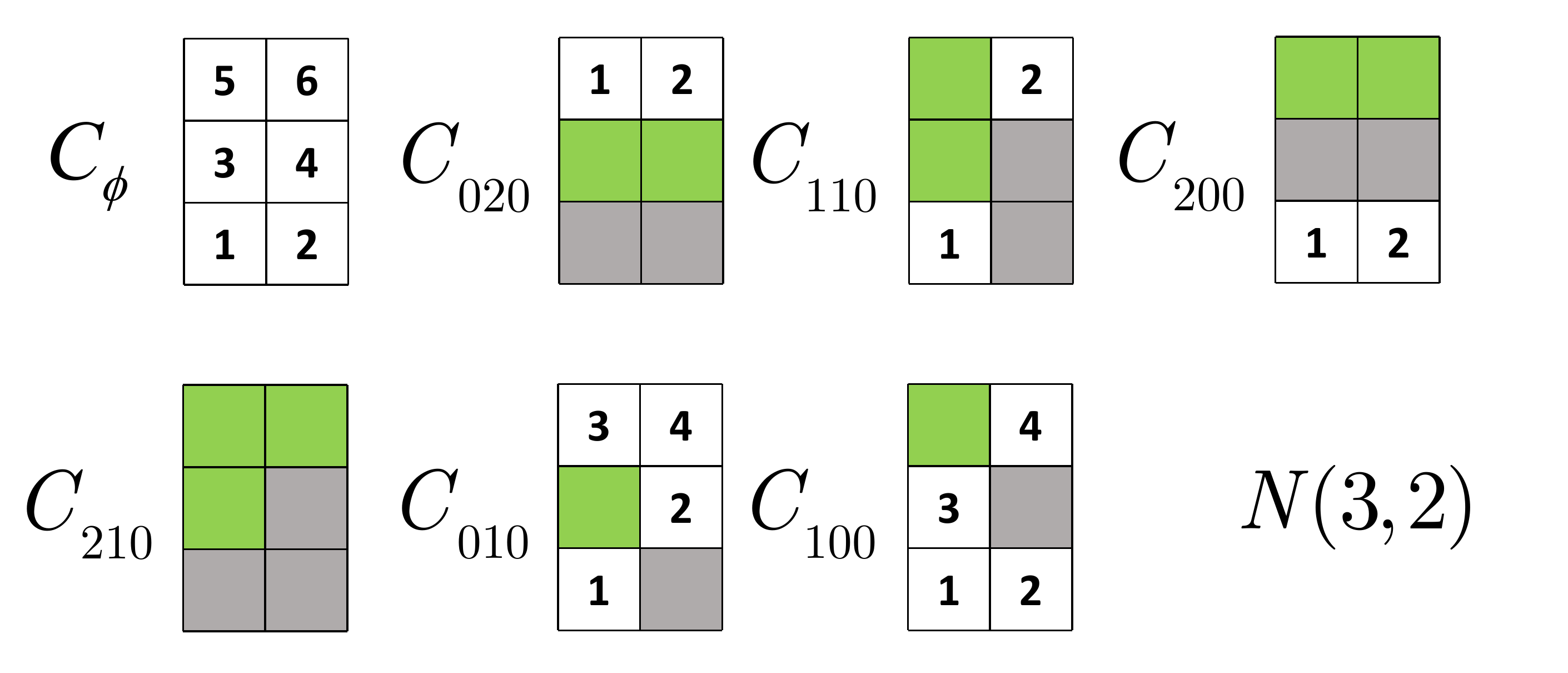}}
\caption{The symmetric chain decomposition of $N(3,2)$.\label{fig:N32}}
\end{figure}

\end{exam}

\section{An induced involution\label{sec:involution}}
There is a natural involution $*$ on $N(m,n)$ defined by $c^*=(n-c_m,n-c_{m-1},\dots, n-c_1)$.
It induces a map $\psi(S(C_\alpha))= (E(C_{\alpha}))^*$, or equivalently
$$ \psi(\alpha)= \alpha^E_{rev}:=(\alpha^E_m,\alpha^E_{m-1},\cdots,\alpha^E_2,0).$$

\begin{theo}\label{theo:involution}
The map $\psi$ is an involution on $S_{m,n}$. More precisely,
$C_{\psi(\alpha)}$ is obtained by rotating the tableau $C_\alpha$ 180 degree.
\end{theo}

\begin{proof}
Suppose Algorithm \ref{alg-CTabCalpah} produces the tableau $C_{\alpha}$. Let $1=q_{1}<q_{2}<\cdots<q_{s}=m$ be all the indices of splitting rows of $C_{\alpha}$.
                 Define $\bar{q_k}:=m-q_{m-k+1}+1$, $k=1,2,\cdots,s$.
                 For any $k\in \{1,2,\cdots,s-1\}$, by Corollary \ref{cor:aEnd}, we obtain $\alpha^E_j=n-\alpha_j$ for each $j=q_{k}+1,q_{k}+2,\cdots,q_{k+1}-1$ and $\alpha^E_{q_{{k+1}}}=\sum\limits_{j=q_{k}}^{q_{k+1}-1}\alpha_j-\sum\limits_{j=q_{k}+1}^{q_{{k+1}}-1}\alpha^E_j$.

                 The reason that $\alpha^E_{rev} \in S_{m,n}$ is as follows:

                                         The condition of $\alpha \in S_{m,n}$ is equivalent to
\begin{align*}
                                             &\alpha_{q_k}+\alpha_{q_k}^E\le n, \quad 1\le k \le s;\\
                                           &\begin{cases}\label{formula3}\tag{3}
                                                                                      \alpha_{q_k} > \alpha^E_{{q_k}+1},\\
                                           \alpha_{q_k} + \alpha_{q_k+1} > \alpha^E_{{q_k}+1} + \alpha^E_{{q_k}+2},\\
                                           \qquad \qquad \vdots\\
                                           \alpha_{q_k} + \alpha_{q_k+1} +\cdots + \alpha_{q_{k+1}-2} > \alpha^E_{{q_k}+1} + \alpha^E_{{q_k}+2} + \cdots + \alpha^E_{q_{k+1}-1},\\
                                           \alpha_{q_k} + \alpha_{q_k+1} +\cdots + \alpha_{q_{k+1}-1} = \alpha^E_{{q_k}+1} + \alpha^E_{{q_k}+2} + \cdots + \alpha^E_{q_{k+1}}.\\
                                           \end{cases}    1\le k \le s-1;
\end{align*}

                 The above formula (\ref{formula3}) holds if and only if the formula
\begin{align*}
                                           &\alpha_{q_k}^E+\alpha_{q_k}\le n, \quad 1\le k \le s;\\
                                           &\begin{cases}\label{formula4}\tag{4}
                                           \alpha^E_{q_{k+1}} > \alpha_{q_{k+1}-1},\\
                                           \alpha^E_{q_{k+1}} + \alpha^E_{q_{k+1}-1} > \alpha_{q_{k+1}-1} + \alpha_{q_{k+1}-2},\\
                                           \qquad \qquad \vdots\\
                                           \alpha^E_{q_{k+1}} + \alpha^E_{q_{k+1}-1} + \cdots + \alpha^E_{{q_k}} > \alpha_{q_{k+1}-1} + \alpha_{q_{k+1}-2} + \cdots + \alpha_{q_k-1},\\
                                           \alpha^E_{q_{k+1}} + \alpha^E_{q_{k+1}-1} + \cdots + \alpha^E_{{q_k}+1} = \alpha_{q_{k+1}-1} + \alpha_{q_{k+1}-2} + \cdots + \alpha_{q_k}.\\
                                           \end{cases} 1\le k \le s-1;
\end{align*}
                 holds. The above formula (\ref{formula4}) holds if and only if the formula
\begin{align*}
                                           &(\alpha_{rev})_{\bar{q_k}}+(\alpha_{rev})_{\bar{q_k}}^E\le n, \quad 1\le k \le s;\\
                                           &\begin{cases}
                                            (\alpha_{rev})_{\bar{q_k}} >(\alpha_{rev})^E_{{\bar{q_k}}+1},\\
                                           (\alpha_{rev})_{\bar{q_k}} + (\alpha_{rev})_{\bar{q_k}+1} > (\alpha_{rev})^E_{\bar{{q_k}}+1} + (\alpha_{rev})^E_{\bar{{q_k}}+2},\\
                                           \qquad \qquad \vdots\\
                                           (\alpha_{rev})_{\bar{q_k}} + (\alpha_{rev})_{\bar{q_k}+1} +\cdots + (\alpha_{rev})_{\bar{q_{k+1}}-2} > (\alpha_{rev})^E_{{\bar{q_k}}+1} + (\alpha_{rev})^E_{{\bar{q_k}}+2} + \cdots + (\alpha_{rev})^E_{\bar{q_{k+1}}-1},\\
                                           (\alpha_{rev})_{\bar{q_k}} + (\alpha_{rev})_{\bar{q_k}+1} +\cdots + (\alpha_{rev})_{\bar{q_{k+1}}-1} = (\alpha_{rev})^E_{\bar{{q_k}}+1} + (\alpha_{rev})^E_{{\bar{q_k}}+2} + \cdots + (\alpha_{rev})^E_{\bar{q_{k+1}}},\\
                                           \text{where $1\le k \le s-1$.}
                                           \end{cases}
\end{align*}
 holds. This is equivalent to $\psi(\alpha)=\alpha^E_{rev} \in S_{m,n}$.
\end{proof}

\begin{exam}
                 We give an example for Theorem \ref{theo:involution}.
                 \begin{figure}[H]
                 \centering{
                 \includegraphics[height=3 in]{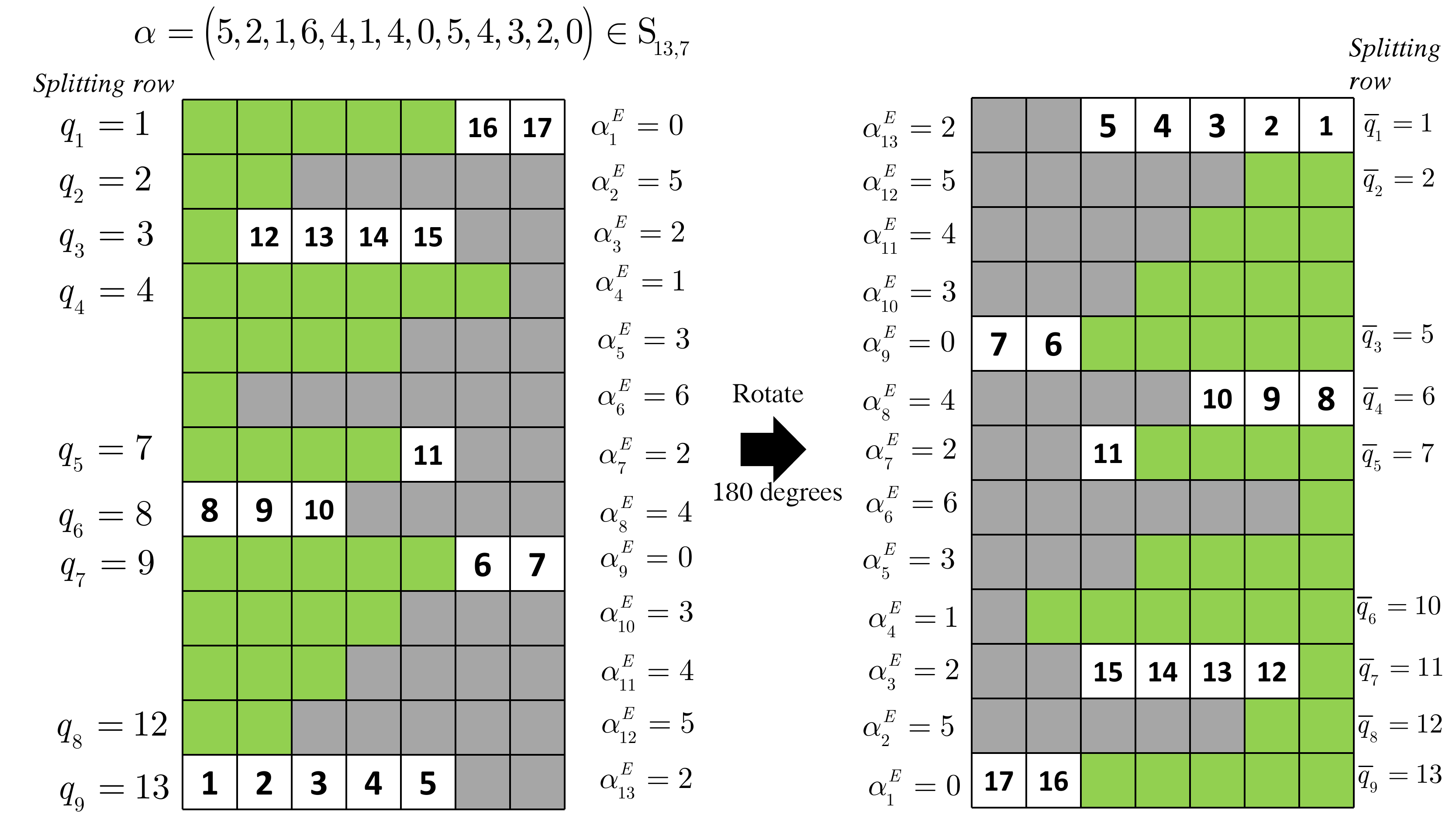}}
                 \caption{Obtaining the $\alpha^E_{rev} \in S_{m,n}$ from $\alpha \in S_{m,n}$ .\label{exa:alphaEndSpoint}}
                 \end{figure}

                 \begin{figure}[H]
                 \centering{
                 \includegraphics[height=3 in]{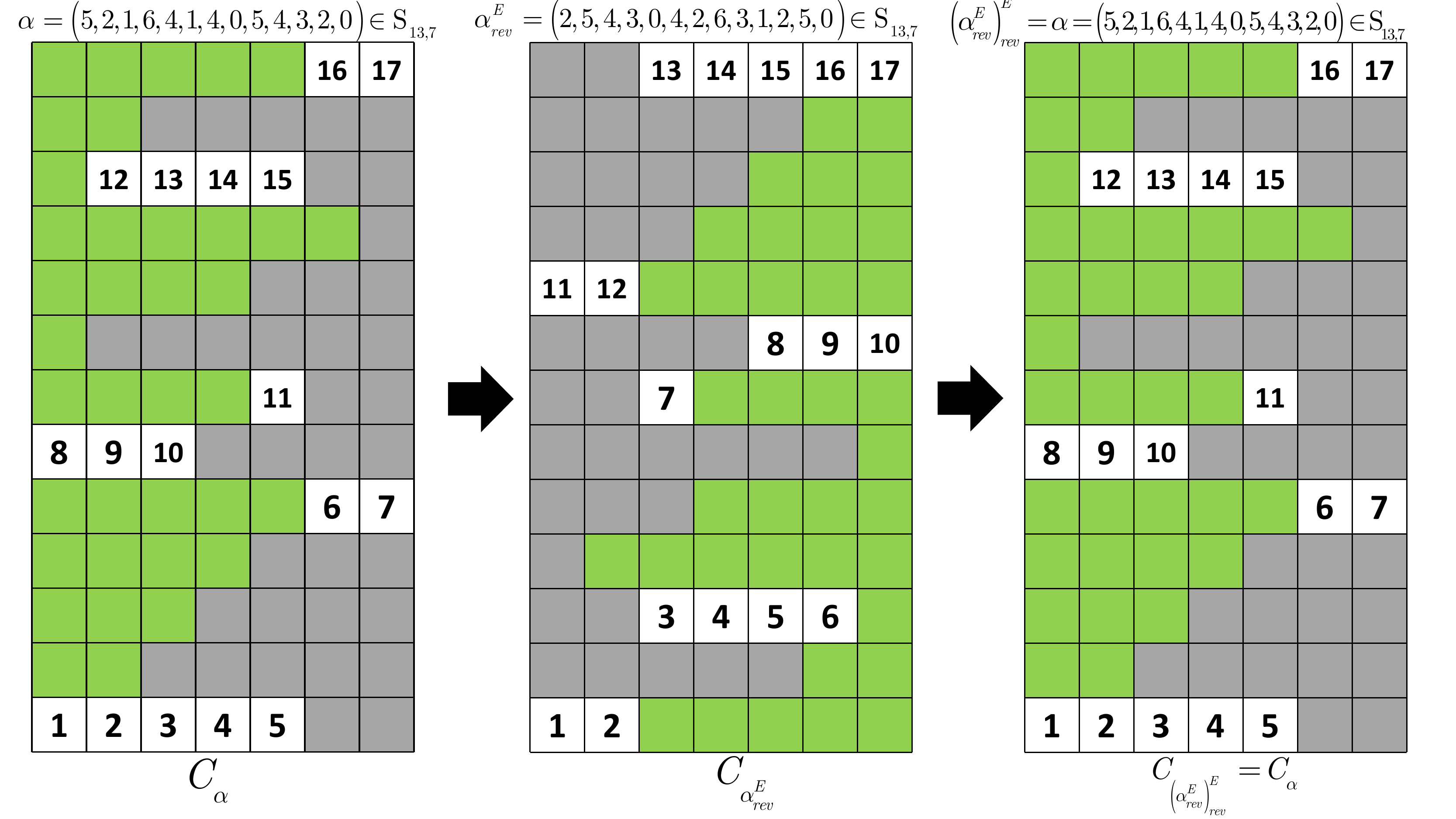}}
                 \caption{$C_{\alpha}=C_{\psi^2(\alpha)}$.\label{exa:alphaInvolution}}
                 \end{figure}
\end{exam}

\textbf{Acknowledgements:} The author would like to thank Professor Guoce Xin for helpful discussions.

\end{document}